\documentclass{amsart}
\usepackage{hyperref}
\usepackage{mathtools}
\usepackage{color}
\usepackage[margin=1in]{geometry}

\newtheorem{theorem}{Theorem}[section]
\newtheorem{proposition}[theorem]{Proposition}
\newtheorem{lemma}[theorem]{Lemma}
\newtheorem{definition}[theorem]{Definition}
\newtheorem{corollary}[theorem]{Corollary}

\begin{document}
\title[Density in Banach Spaces]{Density Properties of Sets in Finite-Dimensional, Strictly Convex Banach Spaces}
\author{Bobby Wilson}
\date{\today}
\address{Department of Mathematics, University of Washington}
\email{blwilson@uw.edu}

\begin{abstract}
In this article, we examine the theorem of Mattila establishing rectifiability for Euclidean regular sets in the setting of strictly convex, finite-dimensional Banach spaces.  
\end{abstract}
\maketitle

\section{Introduction}

Consider a set $E \subset \mathbb{R}^d$.  The classical Lebesgue Density Theorem states that for $\mathcal{L}^d$-almost every $x \in E$, the density
\begin{align*}
\Theta(E,x):=\lim_{r\rightarrow 0^+}  \frac{\mathcal{L}^d[E \cap B(x,r)]}{\mathcal{L}^d[B(x,r)]}
\end{align*}
exists and is equal to 1.  Furthermore, $\Theta(E,x)=0$ for almost every $x \in \mathbb{R}^d \setminus E$.  Considering the notion of Hausdorff measure, one could ask whether or not an analogous version of the density theorem holds for lower-dimensional sets $E$ such that $0<\mathcal{H}^{\alpha}(E)<\infty$ for $0<\alpha\leq d$.  This hypothetical theorem would state that for any real $0<\alpha\leq d$, any $E \subset \mathbb{R}^d$ satisfying $\mathcal{H}^{\alpha}(E)<\infty$ is regular, i.e. 
\begin{align*}
1=\Theta^{\alpha}(E,x):= \lim_{r\rightarrow 0^+}  \frac{\mathcal{H}^{\alpha}[E \cap B(x,r)]}{\omega_{\alpha}r^{\alpha}}
\end{align*} 
for $\mathcal{H}^{\alpha}$  almost every $x \in E$.  However, the local behavior of sets  is more complex when the Hausdorff dimension is less than that of the dimension of the ambient space due to the existence of irregular sets  (positive $\mathcal{H}^{\alpha}$-measure sets for which the density does not exist at almost every point).

In an effort to better characterize this dichotomy, we can identify regularity with rectifiability.  However, we must define rectifiability with respect to integral dimensions.  Thus the first step in this identification consists of showing that regularity not only implies that the density exists at almost every point but also that this density must be defined with respect to an integral Hausdorff measure, $\mathcal{H}^k$ for $k \in \mathbb{Z}_+$. This first step is known as Marstrand's Theorem \cite{marstrand1964varphi}; that is, let
$s$ be a positive number and suppose that there exists a Radon measure $\mu$ on $\mathbb{R}^d$ such that the density $\Theta^s(\mu, a)$ exists and is positive and finite in a set of positive $\mu$ measure. Then $s$ is an integer.

We refer to a Hausdorff $k$-dimensional set for which the $k$-density exists at $\mathcal{H}^k$-almost all of its points as $k$-regular.  Given this definition, it remains to show that $k$-regular sets are $k$-rectifiable.  This was first proved by Marstrand \cite{marstrand1961hausdorff} in the case of 2 dimensional sets in $\mathbb{R}^3$, and  Mattila \cite{mattila1975hausdorff} then extended this to general dimensions.  Preiss \cite{preiss1987geometry} then proved the generalized theorem for measures on $\mathbb{R}^d$ using the notion of tangent measures.

In this article, we examine a generalization of the theorem of Marstrand and Mattila to what are called strictly convex finite-dimensional Banach spaces:  
\begin{theorem}\label{mainthm}
Let $(X,\|\cdot\|)$ be a finite-dimensional Banach space with a strictly convex norm and $E \subset X$ be a $\mathcal{H}^m$-measurable set with $\mathcal{H}^m(E)<\infty$.  $E$ is $m$-rectifiable if and only if $\Theta^m(E,x)=1$ for $\mathcal{H}^m$-almost every $x \in E$.
\end{theorem}

One difficulty in generalizing Mattila's argument is that the geometric observation (Lemma \ref{shrink}) crucial to both the argument of Mattila and the original proof of Marstrand \cite{marstrand1961hausdorff} relies on an intuitive understanding of  Euclidean geometry. The following argument demonstrates that this phenomenon follows simply from the uniform convexity of the Euclidean norm. This allows one to establish a symmetry condition (Lemma \ref{biglemma}) for regular sets, which then leads to a combinatorial condition on a sufficiently large subset. Of course,  since all norms on finite-dimensional linear spaces are equivalent, the combinatorial condition is not affected by the norm structure on a given linear space. This fact allows us to finish the argument simply by appealing directly to the strategy of Mattila and Marstrand.

In the next section, we offer some preliminary definitions and results.  We follow this with a generalization of the important components of Mattila's theorem for Banach spaces with strictly convex norms.  Finally, we discuss two projection lemmas that complete the proof of Theorem \ref{mainthm} in the conclusion. 

\subsection*{Acknowledgement}
The author is supported by NSF grant DMS 1856124, and NSF CAREER Fellowship, DMS 2142064. This material is based upon work supported by the National Science Foundation under Grant No. DMS-1928930 while the author was in residence at the Simons Laufer Mathematical Sciences Institute (formerly MSRI) in Berkeley, California, during the summer of 2023. 
The author would like to thank Tatiana Toro and Max Goering for their gracious support throughout the development and writing of this project.

%%%%%%%%%%%%%%%%%%%%%%%%%%%%%%%%%%%%%%%%%%%%%%%%%%%%%%%%%%%%%%%%%%%%%%%%%%%%%%%%%
%%%%%%%%%%%%%%%%%%%%%%%%%%%%%%%%%%%%%%%%%%%%%%%%%%%%%%%%%%%%%%%%%%%%%%%%%%%%%%%%%

\section{Preliminaries}\label{prelim}

We use a basis representation to reduce any Banach space $(X, \|\cdot\|)$, to $\mathbb{R}^d.$
First, we establish some geometric concepts in $(\mathbb{R}^d, \|\cdot\|)$:
\begin{definition}
The Grassmannian manifold, $G(d,m)$, is the set of all $m$-dimensional linear subspaces of $\mathbb{R}^d$.
\end{definition}

\begin{definition}
 We denote $A(d, m)$ as the space of all $m$-dimensional affine subspaces of $\mathbb{R}^d$.  For every $W\in A(d,m)$ there exists $y \in \mathbb{R}^d$ and $V \in G(d, m)$ such that $W=V+y$.  Furthermore, we denote $A(a,d,m)$ as the collection of all $V \in A(d,m)$ such that $a \in V$. 
\end{definition}

We now define the notion of a strictly convex Banach Space.
\begin{definition}
    Let $(\mathbb{R}^d, \|\cdot\|)$ be a Banach space. We say that $X$ is strictly convex, or that its norm is strictly convex, if
        \begin{align*}
            \|x+y\| < \|x\|+\|y\|
        \end{align*}
    for any nonzero $x, y \in X$ satisfying $x \neq ty$ for all $t \in \mathbb{R}$.
\end{definition}
We note that in the finite-dimensional case, strictly convex and uniformly convex are equivalent.
\begin{definition}
    Let $(\mathbb{R}^d, \|\cdot\|)$ be a Banach space. We say that $X$ is uniformly convex, or that its norm is uniformly convex, if for every $\varepsilon >0$ there exists a $\delta>0$ such that if $\|x-y\|>\varepsilon$ then 
        \begin{align*}
            \|\tfrac{1}{2} x +\tfrac{1}{2} y\| \leq 1-\delta
        \end{align*}
    for $\|x\|\leq 1$ and $\|y\|\leq 1$.
\end{definition}

The fulcrum of the argument for the main theorem is a combinatorial characterization of approximations of the sets in question. In order to properly use this argument, we will need to define the notion of a specific type of linear combinations of sets:

\begin{definition}[Sum and Difference Sets]
Let $E \subset \mathbb{R}^d$ be a Borel set and let $M$ be an integer.  We define $E^{(1)}$ by 
    \begin{align*}
        E^{(1)}:= \left\{ z \in \mathbb{R}^d ~:~ z = x+ (x-y) \mbox{ where } x, y \in E \right\} 
    \end{align*}
For $M\geq 2$, define $E^{(M)}$  inductively by 
	\begin{align*}
		E^{(M)}:=(E^{(M-1)})^{(1)}. 
	\end{align*}
\end{definition}
For any set $E \subset \mathbb{R}^d$ and $y \in \mathbb{R}^d$, we define the distance between $y$ and $E$ by
    \begin{align*}
        \mbox{dist}(E, y)= \mbox{dist}(y, E) := \inf \{ \|x-y\|~:~ x \in E\}
    \end{align*}
and we define the diameter of a set, $E$, by
    \begin{align*}
        |E| := \sup \{ \|x- y\| ~:~ x, y \in E\}
    \end{align*}
For a set $E \subset \mathbb{R}^d$ and a number $r>0$, we let
    \begin{align*}
        N(E, r):= \{ y \in \mathbb{R}^d ~:~ \mbox{dist}(y,E)<r\}.
    \end{align*}

From this notion of diameter we define Hausdorff measure.  For $\delta>0$, $\alpha \geq 0$,
    \begin{align*}
        \mathcal{H}^{\alpha}_{\delta}(E) := \inf \left\{ \sum_i |U_i|^m ~:~ \{U_i\}_{i=1}^{\infty} \mbox{ is an open cover of } E \mbox{ and } |U_i|\leq\delta\right\} 
    \end{align*}
and then 
    \begin{align*}
        \mathcal{H}^{\alpha}(E):= \sup_{\delta>0} \mathcal{H}^{\alpha}_{\delta}(E).
    \end{align*}
We say that $\alpha$ is the dimension of a set $E$ if $\alpha= \sup \{ \beta ~:~ \mathcal{H}^{\beta}(E)=\infty\}= \inf \{ \beta ~:~ \mathcal{H}^{\beta}(E)=0\}$.
We will commonly refer to Hausdorff $\alpha$-dimensional measurable sets as simply $\alpha$-sets.  

Next, we define weak linear approximability.

\begin{definition}
Let $E\subset \mathbb{R}^d$ be a Hausdorff $m$-dimensional set.  We say that $E$ is weakly $m$-linearly approximable if for $\mathcal{H}^m$-almost all $a \in E$ the following holds: if $\eta>0$, there exist $R_0>0$ and $\lambda>0$ such that for any $0<r<r_0$, there is $W \in A(a, d, m)$  such that
    \begin{itemize}
        \item If $a \in E$ and $0<r< R$, then there exists $V \in A(a, d, m)$ such that 
            \begin{align*}
                \mathcal{H}^m(E \cap (B(x, \eta r))\geq \lambda r^m, 
            \end{align*}
            for $x \in W \cap B(a, r)$
        \item 
            \begin{align*}
                \mathcal{H}(E \cap B(a, r) \setminus N(W, \eta r))< \eta r^m. 
            \end{align*}
    \end{itemize}

\end{definition}

The next few theorems are essential structural results that follow directly from their Euclidean analogues.  In fact these results don't follow from the Euclidean arguments as much as they follow from more basic properties of the underlying metric spaces.  All three appear in \cite{mattila1975hausdorff} and \cite{federer2014geometric} (Sections 2.10.17-2.10.19).

\begin{theorem}
Let $(\mathbb{R}^d, \|\cdot \|)$ be a real, finite-dimensional Banach space. If $E \subset \mathbb{R}^d$ and $\mathcal{H}^m(E)<\infty$, then
    \begin{align*}
        \lim_{\delta \rightarrow 0^+} \left[ \sup \frac{\mathcal{H}^m(E\cap S)}{|S|^m} ~:~ x\in S, |S| <\delta\} \right]=1
    \end{align*}
for $\mathcal{H}^m$ a.e. $x \in E$.
\end{theorem}

\begin{theorem}
Let $(\mathbb{R}^d,\|\cdot\|)$ be a real, finite-dimensional Banach space.  If $E \subset \mathbb{R}^d$ and $\mathcal{H}^m(E)<\infty$, then $\Theta^{*m}(E,x)\leq 1$ for $\mathcal{H}^m$ a.e. $x \in \mathbb{R}^d$.  If, in addition, $E$ is $\mathcal{H}^m$-measurable, then $\Theta^m(E,x)=0$ for $\mathcal{H}^m$ a.e. $x \in \mathbb{R}^d \setminus E$.
\end{theorem}

The next statement is a corollary of the former.

\begin{theorem}\label{subsets}
Let $(\mathbb{R}^d,\|\cdot\|)$ be a real, finite-dimensional Banach space.  If $E \subset F \subset \mathbb{R}^d$, $F$ is $m$-regular and $E$ is an $m$-set, then $E$ is $m$-regular.
\end{theorem}

Here regular refers to the existence of densities almost everywhere.

%%%%%%%%%%%%%%%%%%%%%%%%%%%%%%%%%%%%%%%%%%%%%%%%%%%%%%%%%%%%%%%%%%%%%%%%%%%%%%%%%
%%%%%%%%%%%%%%%%%%%%%%%%%%%%%%%%%%%%%%%%%%%%%%%%%%%%%%%%%%%%%%%%%%%%%%%%%%%%%%%%%

\section{Main Arguments}\label{s:main}

\subsection{Tangents to unit balls and Geometry}

\begin{lemma}\label{shrink}
    For every $\varepsilon>0$, there exists $\delta=\delta(\varepsilon)>0$ such that the following holds: if $z, x \in \mathbb{R}^d$, $r>0$, $\eta \in (0, \tfrac{1}{2})$,  $\|x-z\|=(1+\eta)r$, and
        \begin{align*}
            y= z- r\frac{(x-z)}{\|x-z\|} \in \partial B(z, r),
        \end{align*}
    then
        \begin{align*}
            \left| B(z, r) \cup B(x, \eta r)\} \right|- \left| (B(z, r) \cup B(x, \eta r)) \setminus B(y, \varepsilon r)\right|> \delta r.
        \end{align*}

\end{lemma}

\begin{proof}Let $\delta>0$ and without loss of generality, let $z=0$.
    It suffices to obtain a  suitable upper bound on the following set of values:
        \begin{align*}
            &\left\{ \|p- q\|~:~ p \in B(0, r) \setminus B(y_0, \varepsilon r) , q \in B(x_0, \eta r)\right\}.
        \end{align*}
    For any $p \in  B(0, r) \setminus  B(y, \varepsilon r)$, $\|p-y\|\geq\varepsilon r$, $\|p\|\leq r$ and $ \|y\|\leq r$. Uniform convexity implies that there exists $\delta$ such that
    \begin{align*}
        \|p+y\| < 2r(1-\delta).
    \end{align*}
The definition of $y$ implies that $\|x+y\|=\eta r$. Then for any $p \in   B(0, r) \setminus  B(y_0, \varepsilon r)$ and $q \in B(x_0, \eta r)$,
    \begin{align*}
        \|p-q\| \leq \|p+y\|+\|-(y+x)\| +\|x-q\| < 2r(1-\delta) + 2\eta r
    \end{align*}
We know that $\left| B(z, r) \cup B(x, \eta r)\} \right|\geq 2r + 2 \eta r$. Therefore,
    \begin{align*}
        \left| B(z, r) \cup B(x, \eta r)\} \right|- \left| (B(z, r) \cup B(x, \eta r)) \setminus B(y, \varepsilon r)\right|&\geq 2r + 2 \eta r- (2r(1-\delta) + 2\eta r)\\
        &=2 \delta r.
    \end{align*}

\end{proof}

It's important to emphasize that $\delta$ in the previous statement does not depend on $\eta$.
%%%%%%%%%%%%%%%%%%%%%%%%%%%%%%%%%%%%%%%%%%%%%%%%%%%%%%%%%%%%%%%%%%%%%%%%%%%%%%%%%

\subsection{Approximation}

Now we arrive at the primary statement:

\begin{proposition}
Let $(\mathbb{R}^d,\|\cdot\|)$ be a real, finite-dimensional strictly convex Banach space.  Let $E \subset \mathbb{R}^d$ be a $\mathcal{H}^m$-measurable set with $\mathcal{H}^m(E)< \infty$.  $E$ is $m$-rectifiable if and only if $\Theta^m(E,x)=1$ for $\mathcal{H}^m$ almost every $x\in E$. 
\end{proposition}

We begin with a definition.

\begin{definition}
Let $E$ be a $m$-set in a finite-dimensional Banach space $(\mathbb{R}^d,\|\cdot\|)$.  We say that a subset $E_1 \subset E$ is $(\delta,R)$-almost uniform with respect to $E$ if 
	\begin{enumerate}
		\item $\mathcal{H}^m(E \cap S) \leq (1+\delta)  |S|^m$ if $E_1 \cap S \neq \emptyset$ and $|S|<2R$,
		\item $\mathcal{H}^m[E \cap B(a,r)]> (1-\delta)(2r)^m$ if $a \in E_1$ and $0<r<R$.
	\end{enumerate}
\end{definition}

With respect to this definition we have the analogous uniformization lemma:
\begin{lemma} \label{uniform}
Let $(\mathbb{R}^d,\|\cdot\|)$ be a real, finite-dimensional Banach space.  Suppose that $E$ is an $m$-regular subset of $\mathbb{R}^d$, $\epsilon>0$, and $\delta>0$.  Then there are a positive number $R$ and an $m$-set $E_1 \subset E$ such that $\mathcal{H}^m(E\setminus E_1)<\epsilon$ and $E_1$ is $(\delta, R)$-almost uniform with respect to $E$.
\end{lemma}

The first major hurdle is the following lemma.  In order to show that regular sets are linearly approximable, we show that regular sets are almost radial symmetric at every point.  Specifically, we must consider the following statement.

\begin{lemma}\label{biglemma}
    Let $(\mathbb{R}^d,\|\cdot\|)$ be a real, strictly convex, finite-dimensional Banach space, let $E_1\subset E \subset \mathbb{R}^d$, $0<\varepsilon<1$, and $R>0$.  There exists $\delta=\delta(\varepsilon, m)\in (0,\frac{1}{3})$ (depending only on $\varepsilon$) such that if $E_1$ is $(\delta, R)$-almost uniform with respect to $E$, $a, b \in E_1$, and $\|a-b\|<R$, then $2a-b \in N(E, \varepsilon\|a-b\|)$.
\end{lemma}

 The proof here will follow very closely to the original proof of Mattila.

\begin{proof}
Let $\varepsilon>0$, and let $\delta_0=\delta_0(\varepsilon)>0$ be the number given by Lemma \ref{shrink}. Let $0<\delta< \delta_0$ and suppose $E_1\subset E$ is $(\delta, R)$-almost uniform. Finally, let $\eta\in (0, \frac{1}{10})$.

Let $a, b \in E_1$ be such that $0<\|a-b\|=\rho<R$ and define 3 balls in $\mathbb{R}^d$.  The first will be a ball, $A$, centered at $a$ with radius close to but less than the distance between $a$ and $b$.  The second, $B$, will be a small ball centered at $b$ defined so that it doesn't intersect $A$.  The final ball, $C$, will be centered at a point $c$ lying on  both the line containing $a$ and $b$ and the boundary of $A$.  They are defined explicitly as such
    \begin{align*}
        A=B(a, (1-2\eta)\rho) \hspace{.5cm} B=B(b, \eta\rho) \hspace{.5cm} C=B\left(c, \varepsilon(1-2\eta)\rho\right)
    \end{align*}

Our goal is to show that $\mathcal{H}^1(A \cap C \cap E)>0$ which would, in turn, imply that there exists $z \in E$ such that the distance between $z$ and $b'=2a-b$ is less than $\rho\varepsilon$.  To this end, we apply Lemma \ref{shrink} and we get
    \begin{align*}
        |[A\cup B]\setminus C| +\tfrac{1}{2}\delta_0 \rho< |A\cup B|
    \end{align*}
Then for some $c_1>0$,
    \begin{align}\label{imppoint}
        |[A\cup B]\setminus C|^m +c_1\delta_0\rho^m< |A\cup B|^m.
    \end{align}
 The left side of \eqref{imppoint} is bounded below using the hypothesis in the following way
    \begin{align*}
        \frac{1}{1+\delta} \mathcal{H}^m([A\cup B]\setminus C \cap E) +c_1\delta_0\rho^m \leq |[A\cup B]\setminus C|^m +c_1\delta_0\rho^m.
    \end{align*}
In order to bound the right side of \eqref{imppoint}, note that since $A\cap B = \emptyset$, by polynomial expansion, there exists $c_2>0$ such that
    \begin{align*}
        |A\cup B|^m=[|A|+|B|]^m\leq|A|^m+|B|^m+ c_2\delta \rho^m.
    \end{align*}
Again using the hypothesis regarding $E$:

\begin{figure}
\centering
\includegraphics[scale=.3]{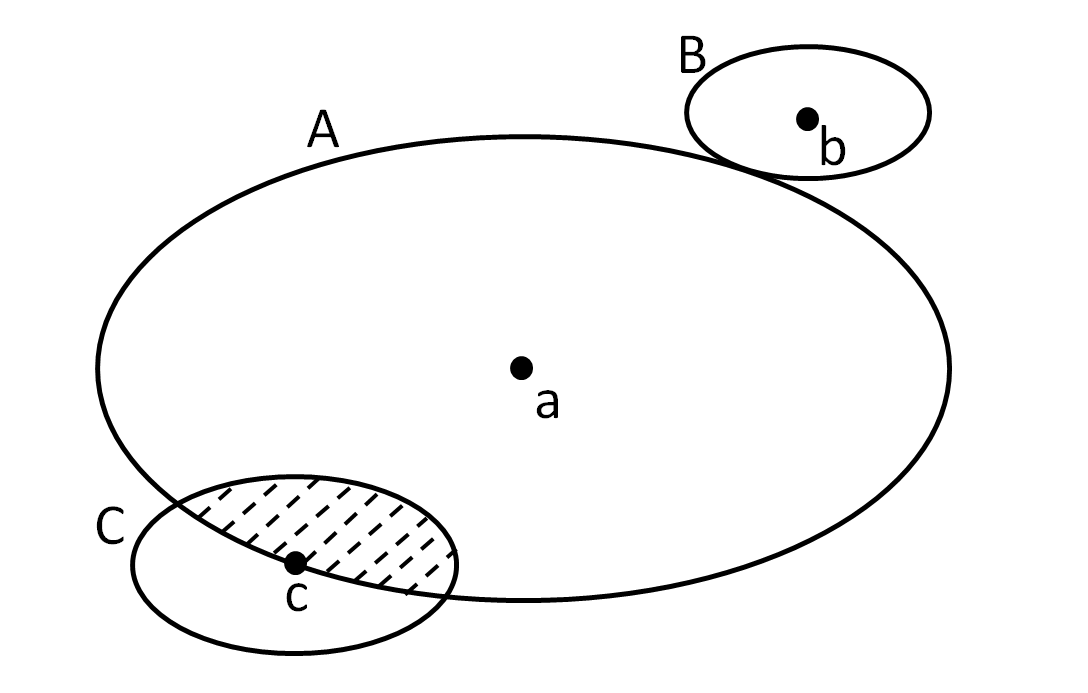}{\centering}  
\caption{\label{balls} Balls shown with an example of a Non-Euclidean norm.  We would like to show that the shaded region is nonempty.}
\end{figure}
    \begin{align*}
        |A|^m+|B|^m+ c_2\eta \rho^m< \frac{1}{1-\delta} \mathcal{H}^m( [A\cup B]\cap E)+c_2\eta \rho^m.
    \end{align*}
and thus 
    \begin{align}\label{inter}
        \frac{1}{1+\delta} \mathcal{H}^m([A\cup B]\setminus C \cap E) +c_1\delta_0\rho^m < \frac{1}{1-\delta} \mathcal{H}^m( [A\cup B]\cap E)+c_2\eta \rho^m.
    \end{align}

By definition of the parameters $(A\cup B)\cap C=A\cap C$.  Using $\mathcal{H}^m([A\cup B]\cap E)- \mathcal{H}^m( (A\cup B)\cap C \cap E)
= \mathcal{H}^m([A\cup B]\setminus C \cap E)$, the fact that $1/(1+\delta)<1$, and rearrangement it follows from equation \eqref{inter} that
    \begin{align*}
        &c_1\delta_0\rho^m + \left[\frac{1}{1+\delta}-\frac{1}{1-\delta} \right] \mathcal{H}^m([A \cup B]\cap E) - c_2\eta \rho^m < \mathcal{H}^m(A\cap C \cap E)\\
        \Leftrightarrow \,&c_1\delta_0\rho^m - \left[\frac{2\delta}{1-\delta^{2}} \right] \mathcal{H}^m([A \cup B]\cap E) - c_2\eta \rho^m < \mathcal{H}^m(A\cap C \cap E)
    \end{align*}
for small enough $\delta>0$ and $\eta>0$, the left side is positive.  
\end{proof}

We can iterate this lemma in the same fashion as Marstrand (\cite{marstrand1961hausdorff} Lemma 2) to obtain the following statement.
\begin{lemma}\label{iterlemma}
    Suppose that $E$ is an $m$-regular subset of $\mathbb{R}^d$, $\eta>0$, $\varepsilon \in (0, 1)$ and $M$ is a positive integer.  Then there are a positive number $d_M$ and an $m$-set $E_M \subset E$ such that $\mathcal{H}^m(E\setminus E_M)<\eta$ and $A^{(M)} \subset N(E, \varepsilon |A|)$ whenever $A \subset E_M$ and $|A|<d_M$.
\end{lemma}

\begin{proof}

We prove by induction. The base case is Lemma \ref{biglemma} with $\frac{1}{2}\varepsilon$ replacing $\varepsilon$. The induction step begins with the supposition that the statement holds for $M=n$. Then given $\eta_0>0$, $\varepsilon_0= \frac{1}{100}\varepsilon >0$, there exists $E_n \subset E$ and $d_n>0$ such that $\mathcal{H}^m(E\setminus E_n)<\eta$ and $A^{(M)} \subset N(E, \varepsilon_0 |A|)$ whenever $A \subset E_n$ and $|A|<d_n$. 

We use Lemma \ref{uniform} to find $E_{n+1} \subset E_{n}$ and $R_n>0$ such that $\mathcal{H}^m(E_n\setminus E_{n+1})<\tfrac{1}{2}\eta_0$ and $E_{n+1}$ is $(\delta, R_n)$-almost uniform with respect to $E$ with $\delta$ small enough to apply Lemma \ref{biglemma}  with $\varepsilon_0$. Now Lemma \ref{biglemma} implies that if we consider  $A \subset E_{n+1}$ and $|A|<\frac{1}{10}\min(R_n, d_n)$, then $A^{(1)} \subset N(E_{n}, \varepsilon_0|A|)$. The symmetry of the metric implies that if $F:= E_n \cap N(A^{(1)}, \varepsilon_0|A|)$, then
    \begin{align*}
        A^{(1)} \subset N(F, \varepsilon_0|A|)
    \end{align*}
and
    \begin{align*}
        |F| \leq |A^{(1)}| +2\varepsilon_0|A| < d_n
    \end{align*}
The induction hypothesis now implies that 
    \begin{align*}
        A^{(n+1)} \subset N(F^{(n)}, \varepsilon_0 |A|) \subset N(E, \varepsilon_0 |F| + \varepsilon_0|A|) \subset N(E, \varepsilon |A|)
    \end{align*}
Choosing $d_{n+1}:= \frac{1}{10}\min(R_n, d_n)$ completes the argument.

\end{proof}

%%%%%%%%%%%%%%%%%%%%%%%%%%%%%%%%%%%%%%%%%%%%%%%%%%%%%%%%%%%%%%%%%%%%%%%%%%%%%%%%%

\subsection{Passing to the Euclidean Case}

Next is the net approximation lemma.  We can now begin to exploit the equivalence of norms in finite-dimensional Banach spaces to reduce the argument of Mattila.  We should establish some notion of equivalence for our given norm: If $\|\cdot\|_2$ represents the Euclidean norm in $\mathbb{R}^d$, then there exists $L>1$ such that
    \begin{align*}
        L^{-1}\|x\| \leq \|x\|_{2} \leq L \|x\|
    \end{align*}
    for all $x \in \mathbb{R}^d.$ If we define
    \begin{align*}
        |E|_2&:= \sup \left\{ \|x-y\|_2 ~:~ x, y \in E \right\}, \\
        B_2(x, r)&:= \{ y \in \mathbb{R}^d~:~ \|x-y\|_2<r\mbox{ and }\\
        N_2(E, r)&:= \left\{x \in \mathbb{R}^d~:~ \inf_{y \in E}\|x-y\|_2 < r \right\}
    \end{align*}
for $x \in\mathbb{R}^d$, $r>0$. We will use the notion of projections throughout the remainder of the article. Since we are considering multiple notions of distance, there are inherently multiple notions of projections that one can consider. However, we will only need to consider orthogonal projections. From here on out, $P_V(E)$ will denote the orthogonal projection of the set, $E$, onto the plane $V$. Finally, we define
     \begin{align*}
        H^{\alpha}_{\delta}(E) := \inf \left\{ \sum_i |U_i|_2^m ~:~ \{U_i\}_{i=1}^{\infty} \mbox{ is an open cover of } E \mbox{ and } |U_i|_2\leq\delta\right\} 
    \end{align*}
    and $H^{\alpha}(E):= \sup_{\delta>0} H^{\alpha}_{\delta}(E).$
Then $|E|_2 \sim_L |E|$ and $N(E, L^{-1}r) \subset N_2(E, r) \subset N(E, Lr)$. This implies that we have the following corollary to Lemma \ref{iterlemma}:

\begin{corollary}\label{itercor}
    Suppose that $E$ is an $m$-regular subset of $\mathbb{R}^d$, $\eta>0$, $\varepsilon \in (0, 1)$ and $M$ is a positive integer.  Then there are a positive number $d_M$ and an $m$-set $E_M \subset E$ such that $H^m(E\setminus E_M)<\eta$ and $A^{(M)} \subset N_2(E, \varepsilon |A|_2)$ whenever $A \subset E_M$ and $|A|_2<d_M$.
\end{corollary}

The conclusion of this corollary is equivalent to Lemma 4.2 in \cite{mattila1975hausdorff}. Therefore, one can similarly deduce Lemma 4.8 from \cite{mattila1975hausdorff}:

\begin{lemma}[Equivalent to Lemma 4.8 from \cite{mattila1975hausdorff}]\label{MainLemma}
If $E$ is an $m$-regular subset of $(\mathbb{R}^d, \|\cdot\|)$, $\eta>0$ and $0<\mu <L^{-2}$, then there is a positive number $R$ and an $m$-set $E^* \subset E$ such that $H^m(E\setminus E^*) < \eta$ and if $a \in E^*$ and $0<r<R$, then

	\begin{itemize}
		\item  there exists  $V \in A(a,d,m)$ such that $E^* \cap [B_2(a,r) \setminus N_2(V, \mu r)]=\emptyset$ and
        \item $V \cap B_2(a, r) \subset N_2(E, \mu r)$.
	\end{itemize}
\end{lemma}

The details of the pathway from Corollary \ref{itercor} to Lemma \ref{MainLemma} follow very closely to  Mattila \cite{mattila1975hausdorff}. However, there are some very important technical obstacles to overcome. Therefore, the next subsection is an attempt to address the technical differences without completely duplicating the work that appears in \cite{mattila1975hausdorff}.  It's important to note before proceeding that the notions of rectifiability, pure unrectifiability, and weak $m$-linear approximability are all equivalent up to changing the norm.

%%%%%%%%%%%%%%%%%%%%%%%%%%%%%%%%%%%%%%%%%%%%%%%%%%%%%%%%%%%%%%%%%%%%%%%%%%%%%%%%%
%%%%%%%%%%%%%%%%%%%%%%%%%%%%%%%%%%%%%%%%%%%%%%%%%%%%%%%%%%%%%%%%%%%%%%%%%%%%%%%%%

\subsection{From Corollary \ref{itercor} to Lemma \ref{MainLemma}}

\begin{lemma}[Lemma 4.4 from \cite{mattila1975hausdorff}]\label{approx}
    Suppose that $k \in \mathbb{Z}$ and $\lambda, p \in \mathbb{R}_+$ such that $1\leq k\leq d$, $0<\lambda <1$ and $p>1$.  Then there exists a positive integer $M=M(k,\lambda, p)$ with the property:

    If $r>0$, $A = \{a_0,...,a_k\} \subset \mathbb{R}^d$, $\|a_i - a_0\|_2\leq r$ and the Euclidean distance between $a_i$ and $\mbox{span} \{a_0, ..., a_{i-1}\}$ is greater than $\lambda r$ for $i=1,...,k$, then $(\mbox{span } A) \cap B_2(a_0,pr) \subset N_2(A^{(M)}, kr).$
\end{lemma}

From here we can show that we can approximate a regular set $E$ by a weakly linearly approximable set.  The first step is the following lemma

\begin{lemma}[Lemma 4.5 from \cite{mattila1975hausdorff}]\label{dim}
    There is a constant $K>1$ depending only on $d$ and $m$ such that for any $m$-regular set $E \subset \mathbb{R}^d$ and for $\eta >0$ there is a positive number $r_0$ and an $m$-set $E_0 \subset E$ with the properties:

	   \begin{enumerate}
	       \item $H^m(E\setminus E_0) < \eta$.
	       \item If $0<r<r_0$, $a \in E_0$ and $T \in A(a,d,m+1)$, then $T \cap B_2(a,Kr) \not\subset N_2(E_0, r)$.
	   \end{enumerate}
\end{lemma}
The proof depends on the fact that $E^{(M)}$ stays close to $E$ and the Hausdorff dimension of $E$ is not enough to cover $B_2(a,Kr)\cap T$.  The following lemma crystallizes this idea
\begin{lemma}
Let $d>m$.  Given a regular $m$-set $E$, we can find a positive number $r_0$ and a $m$-set $E_0 \subset E \subset \mathbb{R}^d$ with the property:

For no positive number $r<r_0$ does the set $N_2(E_0, r)$ contain a (Euclidean) sphere of radius $Kr$, where $K=10^3$.
\end{lemma}

Now Lemmas \ref{iterlemma}, \ref{approx}, and \ref{dim} imply 
\begin{lemma}[Lemma 4.6 from \cite{mattila1975hausdorff}]\label{4.6}
    If $E$ is an $m$-regular subset of $\mathbb{R}^d$, $\eta>0$ and $0<\lambda< 1$, then there are a positive number $d^*$ and an $m$-set $E^* \subset E$ with the properties:
    \begin{enumerate}
        \item $H^m(E \setminus E^*) <\eta.$
        \item If $a \in E^*$ and $0< r< d^*$, then there exists $V \in A(a, d, m)$ such that $E^* \cap [B_2(a, r) \setminus N_2(V, \lambda r)]=\emptyset$
    \end{enumerate}
\end{lemma}

Finally, we have the last  two pieces necessary to prove Lemma \ref{MainLemma}, which importantly do not require a density assumption:

We have the following corollary to the isodiametric inequality (see, for example, \cite{rigot2011isodiametric}):
\begin{lemma}\label{isod}
For any plane, $V \in A(0, d, m)$, $\mathcal{H}^m(V \cap B(0, r)) = (2r)^m.$
\end{lemma}

\begin{lemma}[Lemma 4.7 from \cite{mattila1975hausdorff}]\label{4.7}
    Let $E\subset \mathbb{R}^d$ be an $m$-set and $0<\epsilon<1$. Then there are a positive number, $R$, and an $m$-set $E_0 \subset E$ such that if $a \in \mathbb{R}^d$, $V \in A(a, d, m)$, $B \subset V$ and $0<L^2h< l< R$, then 
    \begin{align*}
        \mathcal{H}^m(E_0 \cap P^{-1}_V(B) \cap N(V, h)) \leq \mathcal{H}^m(E_0 \cap N(B, L^2h) )\leq (1+\epsilon)(1+L^2h/l)^m \mathcal{H}^m(N(B, l) \cap V).
    \end{align*}
\end{lemma}
The proof of the Lemma \ref{4.7} follows in the general case by taking a cylinders adapted to the norm we are considering and using the argument of Marstrand \cite{marstrand1961hausdorff} directly. We include the proof here for completeness:

\begin{proof}[Proof of Lemma \ref{4.7}]
Let $R>0$ and $E_0 \subset E$ be defined so that $\mathcal{H}^m(E \cap S) \leq (1+\epsilon)|S|^m$ for $E_0 \cap S \neq \emptyset$ and $|S| < 2R$.
    We first observe that 
        \begin{align*}
            E_0 \cap P^{-1}_V(B) \cap N(V, h) \subset E_0 \cap P^{-1}_V(B) \cap N_2(V, Lh)\subset E_0  \cap N_2(B, Lh)\subset E_0  \cap N(B, L^2h)
        \end{align*}

    For $r<l<R$ and $x \in V$, define the cylinder
        \begin{align*}
            C(x, r, l):= \{ y \in \mathbb{R}^d~:~ \|P_V(y-x)\|\leq l,\, \|(y-x)- P_V(y-x)\|\leq r\} 
        \end{align*}
    Then $|C(x, r, l)| \leq 2(r+l)$, and if $y \in C(x, r, l)$, then $x \in C(y, r, l)$. Let $\chi_{x, r, l}(y)$ denote the indicator function for $C(x, r, l).$ Then
    \begin{align*}
        (2l)^m \mathcal{H}^m(E_0 \cap N(B, r))&= \int_{E_0 \cap N(B, r)}(2l)^m \,\mathcal{H}^m(dx)= \int_{E_0 \cap N(B, r)}\mathcal{H}^m(B(x, l)\cap V) \,\mathcal{H}^m(dx)\\
        &=\int_{E_0 \cap N(B, r)}\int_{B(x, l)}\,\mathcal{H}^m|_{V}(dy) \,\mathcal{H}^m(dx)\\
        &=\int_{E_0 \cap N(B, r)}\int_{N(B, l) \cap V}\chi_{x, r, l}(y)\,\mathcal{H}^m|_{V}(dy) \,\mathcal{H}^m(dx)\\
        &=\int_{N(B, l) \cap V}\int_{E_0 \cap N(B, r)}\chi_{y, r, l}(x)\,\mathcal{H}^m(dx)\,\mathcal{H}^m|_{V}(dy) \\
        &=\int_{N(B, l) \cap V}\mathcal{H}^m(E_0 \cap C(y, r, l))\,\mathcal{H}^m|_{V}(dy) \\
        & \leq (1+\epsilon)(2(r+l))^m \mathcal{H}^m(N(B, l) \cap V).
    \end{align*}
\end{proof}

We now prove Lemma \ref{MainLemma} which follows primarily from Lemmas \ref{4.6} and \ref{4.7}:

\begin{proof}[Proof of Lemma \ref{MainLemma}]
Let $\delta \in (0, \frac{1}{2})$. Using Lemma \ref{uniform}, Lemma \ref{4.6}, and Lemma \ref{4.7}, there exists $E^* \subset E$ such that $H^m(E \setminus E^*)< \eta$ and $R>0$ such that  
\begin{enumerate}
    \item[(i)] If $a \in E^*$, $V \in A(a, d, m)$, $B \subset V$ and $0<L^2h< l< R$, then 
        \begin{align*}
            \mathcal{H}^m[E^* \cap P^{-1}_VB \cap N(V, h)] \leq (1+\delta)(1+L^2h/l)^m \mathcal{H}^m(N(B, l) \cap V).
        \end{align*}
    \item[(ii)] If $a \in E^*$ and $0<r< R$, then there is $V \in A(a, d, m)$ such that 
        \begin{align*}
            E^* \cap(B_2(a, L^2 r) \setminus N_2(V, L^{-2}\delta^2 \mu r))= \emptyset.
        \end{align*}
    \item[(iii)] If $a \in E^*$ and $0<s< LR$, then 
        \begin{align*}
            \frac{\mathcal{H}^m(E^* \cap B(a, s))}{(2s)^m}>1-\delta
        \end{align*}
\end{enumerate}

Now part (1) of Lemma \ref{MainLemma} follows from part (ii) of the preceding list. Part (ii) also implies that 
    \begin{align*}
        E^* \cap B(a, Lr) \subset   E^* \cap B_2(a, L^2r)\subset  N_2(V, \delta^2L^{-2} \mu r)\subset  N(V, \delta^2L^{-1} \mu r).
    \end{align*}
Suppose, by contradiction, that part (2) of Lemma \ref{MainLemma} does not hold for $a \in E^*$, $V \in A(a, d, m)$ and $r \in (0, R)$ satisfying the conditions above.  Then, there is a point $b \in V \cap B_2(a, r)$ such that $B_2(b,  \mu r) \cap E = \emptyset$. This then implies that $b \in V \cap B(a, Lr)$ and $B(b, L^{-1}\mu r) \cap E =\emptyset$.

Set $B = V \cap B(a, Lr) \setminus B(b, L^{-1} \mu r/2)$. Then, since $\delta^2< 1/2$, triangle inequality implies that 
    \begin{align*}
        E^* \cap N(V, \delta^2L^{-1} \mu r) \cap B(a, Lr) \subset E^* \cap N(V, \delta^2L^{-1} \mu r) \cap P^{-1}_V(B)
    \end{align*}

    Now since $E^* \cap(B(a, Lr) \setminus N(V, \delta^2L^{-1} \mu r))= \emptyset$, part (iii) implies
    \begin{align*}
        \mathcal{H}^m(E^* \cap N(V, \delta^2L^{-1} \mu r) \cap P^{-1}_V(B)) \geq  \mathcal{H}^m(E^* \cap B(a, Lr))>(1-\delta)(2Lr)^m
    \end{align*}
    From the definition of $B$ and Lemma \ref{isod}, 
    \begin{align*}
        \mathcal{H}^m(N(B, \delta L^{-1} \mu r) \cap V) \leq  [(1+\delta\mu L^{-2})2L r]^m- [(\tfrac{1}{2}-\delta)2L^{-1}\mu r]^m
    \end{align*}

    Therefore, part (i) implies
    \begin{align*}
        (1-\delta)(2Lr)^m&< \mathcal{H}^m(E^* \cap B(a, Lr)) \\
        &\leq (1+\delta)(1+L^2\delta^2L^{-1}\mu r/\delta L^{-1}\mu r)^m \left([(1+\delta\mu L^{-2})2L r]^m- [(\tfrac{1}{2}-\delta)2L^{-1}\mu r]^m \right)
    \end{align*}
    Now for $\delta$ small enough, we attain a contradiction.
\end{proof}

%%%%%%%%%%%%%%%%%%%%%%%%%%%%%%%%%%%%%%%%%%%%%%%%%%%%%%%%%%%%%%%%%%%%%%%%%%%%%%%%%
%%%%%%%%%%%%%%%%%%%%%%%%%%%%%%%%%%%%%%%%%%%%%%%%%%%%%%%%%%%%%%%%%%%%%%%%%%%%%%%%%

\section{Projections}\label{s:proj}
   First, the crucial lemma on projections of flat, purely unrectifiable sets:
\begin{lemma}[\cite{mattila1999geometry}, Lemma 16.1] \label{BPT}
    If $E$ is a purely $m$-unrectifiable subset of $\mathbb{R}^d$ and $E$ is weakly $m$-linearly approximable, then $H^m[P_V(E)]=0$ for every $V \in G(d,m)$.
\end{lemma}
Lemma \ref{MainLemma} now provides a significant subset of $E$ for which Lemma \ref{BPT} is relevant. In conjunction with the next statement, we have all that is necessary to complete the proof of Theorem \ref{mainthm}. The proof follows in the same way that Lemma 5.2 in \cite{mattila1975hausdorff}. However, there are a number of careful estimates that one must check because we are not assuming densities exist with respect to the Euclidean norm. Therefore, a contracted argument appears after the statement.
\begin{lemma}\label{RegProj}
    If $E$ is an $m$-regular subset of $\mathbb{R}^d$, then
        \begin{align*}
            \liminf_{r \rightarrow 0^+} \sup_{V \in A(a,d,m)} \frac{H^m(P_V[E \cap B_2(a,r)])}{(2r)^m}\geq 1
        \end{align*}
    for $\mathcal{H}^m$ a.e. $a \in E$.
\end{lemma}

\begin{proof}
Suppose, by contradiction that for some $\eta\in (0, 1)$ 
    \begin{align}\label{eucdensity}
            \liminf_{r \rightarrow 0^+} \sup_{V \in A(a,d,m)} \frac{H^m(P_V[E \cap B_2(a,r)])}{(2r)^m}< \eta.
    \end{align}
  for all $a \in E$. Since $\|\cdot\|$ and $\|\cdot\|_2$ are equivalent with respect to $L>1$ and $E$ is $m$-regular, there exists a constant $c=c(m, d, L)$, a number $R_0>0$ and $E' \subset E$ so that for $r \in (0, R_0)$ and $a \in E'$
    \begin{align*}
        \frac{1}{2}c<\frac{H^m(E \cap B_2(a,r))}{(2r)^m}< 2
    \end{align*}
    Let $t:= ((\eta+1)/2)^{1/m}$, $\epsilon>0$, and $\mu \in ( 0, (1-t)/16)$. Lemma \ref{MainLemma}, then implies that there exists an $m$-regular set $E^*\subset E'$ satisfying the conditions of Lemma \ref{MainLemma}. Moreover, we can conclude that there is a point $a_0 \in E^*$, a positive number $r_0 < R_0$ and $V_0 \in A(a_0, d, m)$ such that the folllowing conditions hold
    \begin{align*}
        &\frac{H^m(E \cap B_2(a_0,r_0))}{(2r_0)^m}<2, \hspace{1.6cm} \frac{H^m((E\setminus E^*) \cap B_2(a_0,r_0))}{(2r_0)^m}<\epsilon\\
       &\frac{H^m(P_V[E^* \cap B_2(a_0,r_0)])}{(2r_0)^m}<\eta, \hspace{1cm}
        E^* \cap [B_2(a_0,r_0) \setminus N_2(V, \mu r_0)]=\emptyset\\
        &V_0 \cap B_2(a_0, r_0) \subset N_2(E^*, \mu r_0)
    \end{align*}
Let $F:= P_V[E^* \cap B_2(a_0,r_0)]$ and $G:= (V \cap B_2(a_0, tr_0))\setminus F.$  Therefore, $F$ is closed and $\frac{H^m(F)}{(2r_0)^m}<\eta$ and 
    \begin{align*}
        \frac{H^m(G)}{(2r_0)^m}\geq t^m-\eta= \frac{1-\eta}{2}
    \end{align*}

We can now cover $G$ with with a finite set of balls, $\{B_2(b_q, \rho_q)\}_{q=1}^l$ such that $b_q \in G$, and
    \begin{align*}
        F \cap \partial B_2(b_q, \rho_q) &\neq \emptyset\\
        F \cap B_2(b_q, \rho_q) &= \emptyset\\
        B_2(b_q, 5\rho_q) \cap B_2(b_{q'}, 5 \rho_{q'}) &= \emptyset\\
        \sum_{q=1}^l \rho_q^m&> K_1r^m_0
    \end{align*}
where $K_1=K_1(m, \eta)$. Then $\mu r_0< (1-t)r_0$, and $V_0 \cap B_2(a_0, r_0) \subset N_2(E^*, \mu r_0)$ implies that $\rho_q \leq \mu r_0$. 

Now consider the sets 
    \begin{align*}
        C_q:= P^{-1}_V(B_2(b_q, \rho_q/2)) \cap N_2(V, (1-t)r_0/2)
    \end{align*}
and label them so that for $q=1, ...., k$, $C_q \cap E' =\emptyset$ and for $q=k+1, ..., l$, there exists $c_q \in C_q \cap E'$. If $q \in \{ k+1, ..., l\}$ and $x \in E \cap B_2(c_q, \rho_q/4)$, then $ x \in B_2(a_0, r_0)$, $P_V(x) \in B_2(b_q, \rho_q)$ and thus $P_V(x) \not\in F$ and $x \not\in E^*$. Thus, 
    \begin{align*}
        \bigcup_{q =k+1}^l E \cap B_2(c_q, \rho_q/4) \subset (E\setminus E^*) \cap B_2(a_0, r_0)
    \end{align*}
Now, since the $B_2(c_q, \rho_q/2)$ are disjoint
    \begin{align*}
        \tfrac{1}{2}4^{-m} \sum_{q=k+1}^l \rho^m_q < \epsilon r^m_0.
    \end{align*}
Then taking $\epsilon< \tfrac{1}{4}4^{-m}$, we get
    \begin{align*}
        \sum_{q=1}^k \rho^m_q>\tfrac{K_1}{2}r^m_0. 
    \end{align*}
For $q \in \{1, ..., k\}$, choose points $e_q \in P_V^{-1}(\partial B_2(b_q, \rho_q)) \cap N_2(V, \mu r_0) \cap E^*$ and $V_q \in A(e_q, d, m)$ such that 
    \begin{align}\label{flatnbhd}
        A_q:= B_2(e_q, \mu^{-1}(1-t)\rho_q/8) \cap V_q \subset N_2(E', (1-t)\rho_q/8)
    \end{align}

Next, $\mu^{-1}\rho_q \leq r_0$ and $(1-t)/8$ is small enough so that $b_q \not\in P_V(A_q)$. Moreover, \eqref{flatnbhd} implies that there exist $K_2=K_2(t)$ and a line segment in $A_q$ that can be covered by a pairwise disjoint family of balls of radius $\rho_q$, centered at points in $E'$, $\{B_2(x^q_j, \rho_q)\}_{j=1}^s$, such that $B_2(x^q_j, \rho_q) \subset B_2(a_0, r_0)$ and $s \geq K_2 \mu^{-1}$. Since $B_2(x^q_j, \rho_q) \subset B_2(a_0, r_0)$, we have
    \begin{align*}
        H^m(E \cap B_2(a_0, r_0)) \geq \sum_{q=1}^k \sum_{j=1}^s H^m( E \cap  B_2(x^q_j, \rho_q))  \geq  c_0\mu^{-1}r^m_0
    \end{align*}
for some $c_0>0$. Therefore, 
    \begin{align*}
        c_0\mu^{-1} \leq \frac{H^m(E \cap B_2(a_0,r_0))}{(2r_0)^m}<2
    \end{align*}
    which provides a contradiction for $\mu$ small enough.
\end{proof}

%%%%%%%%%%%%%%%%%%%%%%%%%%%%%%%%%%%%%%%%%%%%%%%%%%%%%%%%%%%%%%%%%%%%%%%%%%%%%%%%%
%%%%%%%%%%%%%%%%%%%%%%%%%%%%%%%%%%%%%%%%%%%%%%%%%%%%%%%%%%%%%%%%%%%%%%%%%%%%%%%%%

\section{Conclusion}\label{s:argument}

We can finally close the argument for Theorem \ref{mainthm}.  The primary results from the preceding work necessary will be Lemma \ref{MainLemma}, and Theorems \ref{BPT} and \ref{RegProj}.

\begin{proof}[Proof of Theorem \ref{mainthm}]
Let $(\mathbb{R}^d,\|\cdot\|)$ be a $d$-dimensional Banach space.  We begin with the forward direction.  This follows from standard density arguments in combination with the fact that rectifiable sets are linearly approximable.

For the reverse direction, we will combine Lemma \ref{BPT} and Lemma \ref{MainLemma} to contradict Lemma \ref{RegProj}.  Let $0<\eta$, $0<\lambda<1$ and $E \subset \mathbb{R}^d$ be an $m$-regular, purely unrectifiable set.  Furthermore, let $E^*$ and $R$ be given by Lemma \ref{MainLemma}.  Then the conclusions of Lemma \ref{MainLemma} imply that $E^*$ is weakly $m$-linearly approximable with respect to the Euclidean norm. Also, since $E$ is purely unrectifiable, $E^*$ is also purely unrectifiable.  Using the conclusion of Lemma \ref{BPT} we see that $\mathcal{H}^m[P_V(E^*)]=0$ for every $V \in G(d,m)$.  However, since $E^*$ is $m$-regular, Lemma \ref{RegProj} contradicts the previous assertion.
\end{proof}

%%%%%%%%%%%%%%%%%%%%%%%%%%%%%%%%%%%%%%%%%%%%%%%%%%%%%%%%%%%%%%%%%%%%%%%%%%%%%%%%%
%%%%%%%%%%%%%%%%%%%%%%%%%%%%%%%%%%%%%%%%%%%%%%%%%%%%%%%%%%%%%%%%%%%%%%%%%%%%%%%%%

\bibliographystyle{plain}
\bibliography{references}

\end{document}